\documentclass[a4paper,12pt]{amsart}
\usepackage{amsmath,amsfonts,amssymb,amsthm,latexsym,amscd}
\usepackage[dvips]{graphicx}
\usepackage[T1]{fontenc}

\usepackage{parskip}

\newcommand{\B}[1]{{\mathbf #1}}

\newtheorem{thm}{Theorem}[section]
\newtheorem{thm*}{Theorem}
\newtheorem{lem}[thm]{Lemma}
\newtheorem{prop}[thm]{Proposition}
\newtheorem{cor}[thm]{Corollary}

\newtheorem*{q*}{Question}

\theoremstyle{definition}

\newtheorem{rem}[thm]{Remark}
\newtheorem*{rem*}{Remark}
\newtheorem*{rems*}{Remarks}
\newtheorem*{cor*}{Corollary}

\def\B{\mathbf}

\newcommand\Diff{\operatorname{Diff}}

\newcommand\Ham{\operatorname{Ham}}

\newcommand\Aut{\operatorname{Aut}}

\newcommand\vol{\operatorname{vol}}
\newcommand\OP{\operatorname}

\def\Diff{\operatorname{Diff}}

\def\Id{\operatorname{Id}}

\def\a{\alpha}
\def\b{\beta}
\def\area{\operatorname{area}}

\def\ev{\operatorname{ev}}

\def\g{\gamma}
\def\o{\omega}

\begin{document}

\title[The autonomous metric on $\Ham(\B\Sigma_g)$]{On the autonomous metric on groups of Hamiltonian diffeomorphisms of closed hyperbolic surfaces}
\author{Michael Brandenbursky}

\subjclass[2010]{Primary 53; Secondary 57}

\begin{abstract}
Let $\B\Sigma_g$ be a closed hyperbolic surface of genus $g$ and let $\Ham(\B\Sigma_g)$ be the group of Hamiltonian diffeomorphisms of $\B\Sigma_g$. The most natural word metric on this group is the autonomous metric. It has many interesting properties, most important of which is the bi-invariance of this metric. In this work we show that $\Ham(\B\Sigma_g)$ is unbounded with respect to this metric.
\end{abstract}

\maketitle

\section{Introduction and main result}

\label{S:intro}

Let $\B\Sigma_g$ be a closed hyperbolic surface of genus $g$ and let $\o$ be an area form on $\B\Sigma_g$. For every smooth normalized function
$H\colon \B\Sigma_g\to \B R$, i.e. $H$ has a zero mean with respect to $\o$, there exists a unique vector field $X_H$ which satisfies
$$
dH(\cdot)=\o(X_H,\cdot).
$$
It is easy to see that $X_H$ is tangent to the level sets of $H$. Let $h$ be the
time-one map of the flow $h_t$ generated by $X_H$. The diffeomorphism $h$ is
area-preserving and every diffeomorphism arising in this way is called
{\em autonomous}. Such a diffeomorphism is relatively easy to understand
in terms of its generating function.

Denote by $\Ham(\B\Sigma_g)$ the group of Hamiltonian diffeomorphisms of $\B\Sigma_g$. It follows from results of Banyaga that every Hamiltonian diffeomorphism is a composition of finitely many autonomous diffeomorphisms \cite{Ba}.
We define the {\em autonomous norm}
on $\Ham(\B\Sigma_g)$ by
$$
\|f\|_{\OP{Aut}}:=
\min\left\{m\in \B N\,|\,f=h_1\cdots h_m
\text{ where each } h_i \text{ is autonomous}\right\}.
$$
The associated metric is defined by
${\bf d}_{\OP{Aut}}(f,g):=\|fg^{-1}\|_{\OP{Aut}}$.  Since the set of autonomous
diffeomorphisms is invariant under conjugation, the autonomous metric is
bi-invariant. Our main result is the following

\begin{thm*}\label{T:main}
The metric group $(\Ham(\B\Sigma_g), \B d_{\OP{Aut}})$ is unbounded.
\end{thm*}

\begin{rem}
We must notice that Theorem \ref{T:main} may be seen as a corollary from more general results of the author, which will appear in a subsequent paper \cite{B-preparation}. However, we think that the proof given in this paper is much easier than the one presented in \cite{B-preparation}. 
\end{rem}

\section{Preliminaries}
\label{sec-preliminaries}
In this paper we consider only normalized functions $\B\Sigma_g\to\B R$.

\subsection{Quasi-morphisms}
\label{ssec-qm}

A function $\psi\colon G \to~\B R$ from a group $G$ to the reals is called a {\em quasi-morphism}
if there exists a real number $A\geq 0$ such that
$$
|\psi(gg') - \psi(g) - \psi(g')|\leq C
$$
for all $g,g'\in G$. The infimum of such $C$'s is called the
\emph{defect} of $\psi$ and is denoted by $D_\psi$.
If $\psi(g^n)=n\psi(g)$ for all $n\in \B Z$
and $g\in G$ then $\psi$ is called \emph{homogeneous}. Any
quasi-morphism $\psi$ can be homogenized by setting
$$\overline{\psi} (g) := \lim_{p\to +\infty} \frac{\psi (g^p)}{p}.$$
The vector space of homogeneous quasi-morphisms on $G$
is denoted by $Q(G)$. The space of homogeneous quasi-morphisms on $G$ modulo the space of homomorphisms on $G$ is denoted by $\widehat{Q}(G)$. For more information about quasi-morphisms and their connections to different brunches of mathematics, see \cite{Calegari}.

\subsection{Polterovich construction}
\label{ssec-polt-construction}

Let $\B M$ be an oriented closed Riemannian manifold equipped with a Riemannian volume form, and denote by $\Diff_0(\B M, \vol)$ the identity component of the group of volume preserving diffeomorphisms of $\B M$. Let us describe a construction, due to L. Polterovich, of quasi-morphisms on
$\Diff_0(\B M, \vol)$.

Let $z\in \B M$. Suppose that the group $\pi_1(\B M,z)$ has a trivial center and it admits a \emph{non-trivial} homogeneous quasi-morphism $$\psi\colon\pi_1(\B M,z)\to\B R.$$
For each $x\in\B M$ let us choose a short geodesic path from $x$ to $z$. In \cite{P} Polterovich constructed the induced \emph{non-trivial} homogeneous quasi-morphism $\overline{\Psi}$ on $\Diff_0(\B M, \vol)$ as follows:\\
For each $x\in\B M$ and an isotopy $\{g_t\}_{t\in[0,1]}$ between $\Id$ and $g$, let $g_x$ be a closed loop in $\B M$ which is a concatenation of a geodesic path from $z$ to $x$, the path $g_t(x)$ and a described above geodesic path from $g(x)$ to~$z$. Denote by $[g_x]$ the corresponding element in $\pi_1(\B M,z)$ and set
$$\Psi(g):=\int\limits_{\B M} \psi([g_x])\vol\qquad\qquad
\overline{\Psi}(g):=\lim\limits_{p\to\infty}\frac{1}{p}\int\limits_{\B M} \psi([(g^p)_x])\vol.$$
The maps $\Psi$ and $\overline{\Psi}$ are well-defined quasi-morphisms because every diffeomorphism in $\Diff_0(\B M, \vol)$ is volume-preser\-ving. They do not depend on a choice of the path $\{g_t\}$ because the evaluation map
$$\ev\colon\Diff_0(\B M, \vol)\to \B M,$$
where $\ev(g)=g(z)$, induces a map between $\pi_1(\Diff_0(\B M, \vol),\Id)$ and $\pi_1(\B M,z)$ whose image lies in the center of $\pi_1(\B M,z)$, which is trivial by our assumption. In addition, the quasi-morphism $\overline{\Psi}$ neither depends on the choice of a family of geodesic paths, nor on the choice of a base point $z$. For more details see \cite{P}. We abuse the notation and sometimes write $\pi_1(\B M)$ instead of
$\pi_1(\B M, z)$.

Denote by $\mathfrak{Polt}_{\B M}$ the linear map, induced by Polterovich construction, from
$\widehat{Q}(\pi_1(\B M))\to \widehat{Q}(\Diff_0(\B M, \vol))$. Since every homogeneous quasi-morphism on $\pi_1(\B M)$, which is not a homomorphism, defines a homogeneous quasi-morphism on $\Diff_0(\B M, \vol)$, which is also not a homomorphism, we obtain the following

\begin{cor}\label{cor:Polt-map-injectivity}
The linear map $\mathfrak{Polt}_{\B M}\colon \widehat{Q}(\pi_1(\B M))\to \widehat{Q}(\Diff_0(\B M, \vol))$ is injective.
\end{cor}

Let $g>1$. Since the group $\Ham(\B \Sigma_g)$ is simple \cite{Ba,Ba1}, the linear space $\widehat{Q}(\Ham(\B \Sigma_g))$ coincides with
$Q(\Ham(\B \Sigma_g))$. We denote by $\mathfrak{Polt}_g$ the linear map $\widehat{Q}(\pi_1(\B \Sigma_g))\to \widehat{Q}(\Ham(\B \Sigma_g))$ induced by Polterovich construction.

\begin{prop}\label{P:qm-injectivity}
The linear map $\mathfrak{Polt}_g\colon \widehat{Q}(\pi_1(\B \Sigma_g))\to \widehat{Q}(\Ham(\B \Sigma_g))$ is injective.
\end{prop}

\begin{proof}
Suppose that $\mathfrak{Polt}_g$ is not injective. Then there exists a quasi-morphism $\psi$ in $Q(\pi_1(\B \Sigma_g))$, which is not a homomorphism, such that the induced homogeneous quasi-morphism $\overline{\Psi}$ on $\Diff_0(\B \Sigma_g, \area)$ vanishes on the group $\Ham(\B \Sigma_g)$.

\begin{lem}\label{L:qm-comm-gp-ext}
Let $G$ be a group and $G'$ its commutator subgroup. Then every homogeneous quasi-morphism on $G$ that vanishes on $G'$ is a homomorphism.
\end{lem}

\begin{proof}
Let $\varphi\colon G\to\B R$ be a homogeneous quasi-morphism such that $\varphi(g')=0$ for every $g'\in G'$. Since $G'$ is a normal subgroup of $G$, the map $\widehat{\varphi}\colon G/G'\to \B R$, given by $\widehat{\varphi}(gG')=\varphi(g)$, is a well-defined homogeneous quasi-morphism. The group $G/G'$ is abelian and hence $\widehat{\varphi}$ is a homomorphism and so is $\varphi$.
\end{proof}

By a theorem of Banyaga \cite{Ba1} the group $\Ham(\B \Sigma_g)$ coincides with the commutator subgroup of $\Diff_0(\B \Sigma_g, \area)$. It follows from Lemma \ref{L:qm-comm-gp-ext} that $\overline{\Psi}$ is a homomorphism, which contradicts the fact that the map
$$\mathfrak{Polt}_{\B \Sigma_g}\colon \widehat{Q}(\pi_1(\B\Sigma_g))\to \widehat{Q}(\Diff_0(\B \Sigma_g, \area))$$
is injective by Corollary \ref{cor:Polt-map-injectivity}.
\end{proof}

\section{Proofs}
\label{sec-proofs}

\subsection{Curves traced by Morse autonomous flows}
\label{sec-curves-aut-flows}
Let $h_t$ be an autonomous flow generated by a Morse function $H\colon\B\Sigma_g\to\B R$ and set $h:=h_1$. Let $x\in\B\Sigma_g$ which satisfies the following conditions:

\begin{itemize}
\item
$x$ is a regular point of $H$,
\item
$x$ belongs to only one connected component, i.e. a simple closed curve in $\B\Sigma_g$, of the set $H^{-1}(H(x))$.
\end{itemize}

Such a set of points in $\B\Sigma_g$ is denoted by $\mathrm{Reg}_H$. Note that the measure of $\B\Sigma_g\setminus\mathrm{Reg}_H$ is zero.
For each $x\in\mathrm{Reg}_H $ let $c_x\colon[0,1]\to\B\Sigma_g$
be an injective path (on $(0,1)$), such that $c_x(0)=c_x(1)=x$ and its image is a simple closed curve which is a connected component of $H^{-1}(H(x))$. For every $y_1,y_2\in \B\Sigma_g$ choose an injective map $s_{y_1y_2}\colon[0,1]\to \B\Sigma_g$
whose image is a short geodesic path from $y_1$ to $y_2$. Define
\begin{equation}\label{eq:gamma-reg}
\gamma_x(t):=
\begin{cases}
s_{zx}(3t) &\text{ for } t\in \left [0,\frac13\right ]\\
c_x(3t-1) &\text{ for } t\in \left [\frac13,\frac23\right ]\\
s_{xz}(3t-2) & \text{ for } t\in \left [\frac23,1\right ]
\end{cases}.
\end{equation}
Denote by $[\gamma_x]$ the corresponding element in $\pi_1(\B\Sigma_g, z)$. Let $x\in\mathrm{Reg}_H$ and let $[h_x]$ be an element in $\pi_1(\B\Sigma_g, z)$ represented by a path which is a concatenation of paths $s_{zx}$, $h_t(x)$ and $s_{xz}$. Then for each $p\in\B N$ the element $[h^p_x]$ can be written as a product
\begin{equation}\label{eq:braid-morse-aut}
[h^p_x]=\a'_{p,x}\circ[\gamma_{x}]^{k_{h,p}}\circ\a''_{p,x}\thinspace,
\end{equation}
where $k_{h,p}$ is an integer which depends only $h$, $p$ and $x$, and the word length of elements $\a'_{p,x}\thinspace, \a''_{p,x}$ in $\pi_1(\B\Sigma_g, z)$ is bounded by some constant $K$ which is independent of $h$, $x$ and $p$.

Denote by $\mathcal{MCG}_g^1$ the mapping class group of a surface $\B\Sigma_g$ with one puncture $z$. Recall that there is a following short exact sequence due to Birman \cite{Bir1}
\begin{equation}\label{eq:Birman-SES}
1\to \pi_1(\B\Sigma_g, z)\to\mathcal{MCG}_g^1\to\mathcal{MCG}_g\to 1,
\end{equation}
where $\mathcal{MCG}_g$ is the mapping class group of a surface $\B\Sigma_g$. Hence we view $\pi_1(\B\Sigma_g, z)$ as a normal subgroup of $\mathcal{MCG}_g^1$.

\begin{prop}\label{P:finite-conj-classes}
Let $g>1$. There exists a finite set $S_g$ of elements in $\mathcal{MCG}_g^1$, such that for every Morse function $H\colon\B\Sigma_g\to\B R$ and every $x\in\mathrm{Reg}_H$ the loop $[\g_x]\in\pi_1(\B\Sigma_g, z)<\mathcal{MCG}_g^1$
is conjugated to some element in $S_g$.
\end{prop}

\begin{proof} Let $x\in\mathrm{Reg}_H$. If the loop $\g_x(t)$ is homotopically trivial in $\B\Sigma_g$, then $[\g_x]=1_{\mathcal{MCG}_g^1}$.
Suppose that $\g_x(t)$ is homotopically non-trivial in $\B\Sigma_g$. We say that simple closed curves $\delta,\delta'\in\B\Sigma_g$ are equivalent $\delta\cong\delta'$, if there exists a homeomorphism $f\colon\B\Sigma_g\to\B\Sigma_g$ such that $f(\delta)=\delta'$. It follows from classification of surfaces that the set of equivalence classes $\mathcal{E}_g$ is finite. Let $c_x$ be a simple closed curve defined in \eqref{eq:gamma-reg}. Since $\B\Sigma_g$ and $c_x$ are oriented, the curve $c_x$ splits in $\B\Sigma_g\setminus\{x\}$ into two simple closed curves $c_{x,+}$ and $c_{x,-}$ which are homotopic in $\B\Sigma_g$, see Figure \ref{fig:path-split}.
\begin{figure}[htb]
\centerline{\includegraphics[height=1.3in]{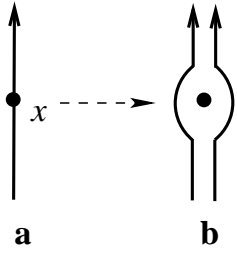}}
\caption{\label{fig:path-split} Part of the curve $c_x$ is shown in Figure \textbf{a}. Its splitting into curves $c_{x,+}$ and $c_{x,-}$ is shown in Figure \textbf{b}. The left curve in Figure \textbf{b} is $c_{x,+}$ and the right curve is $c_{x,-}.$}
\end{figure}

The image of the element $[\g_x]$ in $\mathcal{MCG}_g^1$, under the Birman embedding \eqref{eq:Birman-SES} of $\pi_1(\B\Sigma_g,z)$ into $\mathcal{MCG}_g^1$, is conjugated to $t_{c_{x,+}}\circ t^{-1}_{c_{x,-}}$, where $t_{c_{x,+}}$ and $t_{c_{x,-}}$ are Dehn twists in $\B\Sigma_g\setminus\{x\}$ about curves $c_{x,+}$ and $c_{x,-}$ respectively.

Note that if $c_x\cong\delta$ then there exists a homeomorphism $f\colon\B\Sigma_g\to\B\Sigma_g$ such that $f(c_x)=\delta$, hence $f(c_{x,+})=\delta_+$ and $f(c_{x,-})=\delta_-$. We have
$$
t_{\delta_+}=f\circ t_{c_{x,+}}\circ f^{-1}\qquad t_{\delta_-}=f\circ t_{c_{x,-}}\circ f^{-1}.
$$
This yields
$$f\circ(t_{c_{x,+}}\circ t^{-1}_{c_{x,-}})\circ f^{-1}=t_{\delta_+}\circ t^{-1}_{\delta_-}.$$
In other words, an element $[\g_x]$ is conjugated in $\mathcal{MCG}_g^1$ to some $t_{\delta_+}\circ t^{-1}_{\delta_-}$, where $\delta$ is a representative of an equivalence class in $\mathcal{E}_g$. Let $\{\delta_i\}_{i=1}^{\#\mathcal{E}_g}$ be a set of simple closed curves in $\B\Sigma_g$, such that each equivalence class in $\mathcal{E}_g$ is represented by some $\delta_i$. Let
\begin{equation}\label{eq:finite-set-S}
S_g:=\{t_{\delta_{1,+}}\circ t^{-1}_{\delta_{1,-}},\ldots,t_{\delta_{\#\mathcal{E}_g,+}}\circ t^{-1}_{\delta_{\#\mathcal{E}_g,-}}\}.
\end{equation}
It follows that $[\g_x]$ is conjugated to some element in $S_g$. Noting that the set $S_g$ neither depend on $H$ nor on $x$, we conclude the proof of the proposition.
\end{proof}

\subsection{Continuity of Polterovich quasi-morphisms}
\label{ssec-continuity}
The aim of this subsection is to prove the following technical result which
will be used in the proof of Theorem \ref{T:main}.

\begin{thm}\label{T:Morse-Ham}
Let $H\colon\B\Sigma_g\to\B R$ and $\{H_k\}_{k=1}^\infty$ be a sequence of functions such that
each $H_k\colon\B\Sigma_g\to\B R$ and $H_k\xrightarrow[k\rightarrow\infty]{}H$ in
$C^1$-topology. Let $h$ and $h_k$ be the time-one maps of the Hamiltonian flows
generated by $H$ and $H_k$ respectively. Then
$$
\lim\limits_{k\to\infty}\overline{\Psi}(h_k)=\overline{\Psi}(h),
$$
where $\overline{\Psi}$ is a homogeneous quasi-morphism on $\Ham(\B\Sigma_g)$ induced by Polterovich construction.
\end{thm}

\begin{proof}
The proof of this theorem is similar to the proof of Theorem 3.4 in \cite{BK2}. For the reader convenience, we present the proof below.

At this point we recall a definition of the $L^1$-metric on the group $\Ham(\B\Sigma_g)$.
It is defined as follows. Let
$$
\mathcal{L}_1\{h_t\}:=
\int_0^1 \int_{\B \Sigma_g}|\dot{h_t}(x)|\o dt
$$
be the $L^1$-length
of a path $\{h_t\}_{t\in [0,1]}\in\Ham(\B \Sigma_g)$,
where $|\dot{h_t}(x)|$ denotes the length of the tangent
vector $\dot{h_t}(x)\in T_x\B\Sigma_g$ induced by the Riemannian
metric. Observe that this length is right-invariant, that is,
$\mathcal{L}_1\{h_t\circ f\}=\mathcal{L}_p\{h_t\}$
for any $f\in \Ham(\B\Sigma_g)$. It defines a non-degenerate right-invariant
metric on $\Ham(\B\Sigma_g)$ by
$$
{\bf d}_1(h_0,h_1):=\inf_{h_t}\mathcal{L}_1\{h_t\},
$$
where the infimum is taken over all paths from
$h_0$ to $h_1$. See Arnol'd-Khesin \cite{AK} for a detailed discussion.
We set $\|h\|_1:={\bf d}_1(\Id,h)$. In \cite{B-infinity} the author proved the following

\begin{thm}[\cite{B-infinity}]\label{T:lipschitz}
Let $\B \Sigma_g$ be a closed hyperbolic surface, and $\overline{\Psi}$ a homogeneous quasi-morphism on $\Ham(\B\Sigma_g)$ induced by Polterovich construction. Then $\overline{\Psi}$ is Lipschitz with respect to the $L^1$-metric on the group $\Ham(\B\Sigma_g)$, i.e. there exists $C>0$ such that $\forall h\in \Ham(\B\Sigma_g)$
$$\overline{\Psi}(h)\leq C\|h\|_1.$$
\end{thm}
\begin{lem}\label{lem:epsilon-close}
Let $G\colon\B\Sigma_g\to\B R$ be smooth function. Then for any $\epsilon>0$ and $p\in\B N$ there exists $\delta_p>0$, such that if
$G$ is $\delta_p$-close to a smooth function $F\colon\B\Sigma_g\to\B R$ in $C^1$-topology, then
$$\B d_1(g^p,f^p)<\epsilon,$$
where $g_t$ and $f_t$ are the Hamiltonian flows generated by $G$ and $F$, and $g$ and $h$ are time-one maps of these flows.
\end{lem}
\begin{proof}
We replace $\B D^2$ by $\B\Sigma_g$ in the proof of Lemma 3.7 in \cite{BK2}. Now the proof is identical to the proof of Lemma 3.7 in \cite{BK2}.
\end{proof}

\begin{prop}\label{P:Morse-Ham-epsilon}
Let $H\colon\B\Sigma_g\to\B R$. Then for any $\epsilon>0$ there exists $\delta>0$,
such that if $F\colon\B\Sigma_g\to\B R$ is $\delta$-close to $H$ in $C^1$-topology then:
$$
\left|\overline{\Psi}(h)-\overline{\Psi}(f)\right|\leq\epsilon,
$$
where $h$ and $f$ are time-one maps of flows generated by $H$ and $F$.
\end{prop}

\begin{proof}
Fix some $\epsilon>0$. Let $C$ be the constant which was defined in Theorem \ref{T:lipschitz}.
Take $p\in\B N$ such that $\frac{D_{\overline{\Psi}}+C}{p}<\epsilon$. It
follows from Lemma \ref{lem:epsilon-close} that there exists $\delta_p>0$, such
that if $F$ is $\delta_p$-close to $H$ in $C^1$-topology, then
$\B d_1(f^p,h^p)<1$. Thus we obtain
$$
\left|\overline{\Psi}(f)-\overline{\Psi}(h)\right|=
\frac{1}{p}\left|\overline{\Psi}(f^p)-\overline{\Psi}(h^p)\right|\leq
\frac{D_{\overline{\Psi}}+\left|\overline{\Psi}(f^ph^{-p})\right|}{p}.
$$
It follows from Theorem \ref{T:lipschitz} that
$$
\left|\overline{\Psi}(f^ph^{-p})\right|\leq C\B d_1(Id,f^ph^{-p})=C\B d_1(f^p,h^p)<C.
$$
Thus
$$
\left|\overline{\Psi}(f)-\overline{\Psi}(h)\right|<\frac{D_{\overline{\Psi}}+C}{p}<\epsilon.
$$
\end{proof}

Proposition \ref{P:Morse-Ham-epsilon} concludes the proof of Theorem \ref{T:Morse-Ham}.
\end{proof}

\subsection{Proof of Theorem \ref{T:main}}
\label{ssec-proof-main}

Recall that the group $\pi_1(\B\Sigma_g)$ is a subgroup of $\mathcal{MCG}_g^1$. Denote by $Q_{\mathcal{MCG}_g^1}(\pi_1(\B\Sigma_g), S_g)$ the space of homogeneous quasi-morphisms on $\pi_1(\B\Sigma_g)$ so that:
\begin{itemize}
\item
For each $\varphi\in Q_{\mathcal{MCG}_g^1}(\pi_1(\B\Sigma_g), S_g)$ there exists $\widehat{\varphi}\in Q(\mathcal{MCG}_g^1)$ such that $\widehat{\varphi}|_{\pi_1(\B\Sigma_g)}=\varphi$,
\item
each $\varphi$ vanishes on the finite set $S_g$,
\end{itemize}
where $S_g$ is the set defined in \eqref{eq:finite-set-S}. The group $\pi_1(\B\Sigma_g)$ contains a non-abelian free group, and thus is not virtually abelian. It is an infinite normal subgroup of $\mathcal{MCG}_g^1$ and hence is a non-reducible subgroup of $\mathcal{MCG}_g^1$, see
\cite[Corollary 7.13]{Ivanov}. Now, by a result of Bestvina-Fujiwara \cite[Theorem 12]{BF} we have the following

\begin{cor}\label{cor:inf-dim-qm-restriction}
The space $Q_{\mathcal{MCG}_g^1}(\pi_1(\B\Sigma_g), S_g)$ is infinite dimensional.
\end{cor}

Note that since every non-trivial quasi-morphism in $Q_{\mathcal{MCG}_g^1}(\pi_1(\B\Sigma_g), S_g)$ vanishes on $S_g$, it can not be a homomorphism.  Hence the space $Q_{\mathcal{MCG}_g^1}(\pi_1(\B\Sigma_g), S_g)$ may be viewed as a linear subspace of $\widehat{Q}(\pi_1(\B\Sigma_g))$. Recall that by Corollary \ref{cor:Polt-map-injectivity} the linear map
$$\mathfrak{Polt}_g\colon \widehat{Q}(\pi_1(\B \Sigma_g))\hookrightarrow \widehat{Q}(\Ham(\B \Sigma_g))$$
is injective. Hence the map
$$\mathfrak{Polt}_g\colon Q_{\mathcal{MCG}_g^1}(\pi_1(\B\Sigma_g), S_g)\hookrightarrow \widehat{Q}(\Ham(\B \Sigma_g))$$
is also injective.

Denote by $Q(\Ham(\B \Sigma_g),\Aut)$ the space of homogeneous quasi-morphisms on $\Ham(\B \Sigma_g)$ that vanish on the set
$\OP{Aut}\subset \Ham(\B\Sigma_g)$ of all autonomous diffeomorphisms. Since $Q(\Ham(\B \Sigma_g),\Aut)$ contains no non-trivial homomorphisms, it is viewed as a linear subspace of $\widehat{Q}(\Ham(\B \Sigma_g))$. Now we a ready to state and prove our key proposition.

\begin{prop}\label{P:key-proposition}
The image of the map
$$\mathfrak{Polt}_g\colon Q_{\mathcal{MCG}_g^1}(\pi_1(\B\Sigma_g), S_g)\hookrightarrow \widehat{Q}(\Ham(\B \Sigma_g))$$
lies in the linear space $Q(\Ham(\B \Sigma_g),\Aut)$.
\end{prop}

\begin{proof}
Let $\psi\in Q_{\mathcal{MCG}_g^1}(\pi_1(\B\Sigma_g), S_g)$ and $h\in \Ham(\B \Sigma_g)$ an autonomous diffeomorphism. We need to show that $\overline{\Psi}(h)=0$, where $\overline{\Psi}=\mathfrak{Polt}_g(\psi)$. Since Morse functions on $\B \Sigma_g$ form a dense subset in the set of all smooth functions in $C^1$-topology \cite{Mil}, by Theorem \ref{T:Morse-Ham} it is enough to show that $\overline{\Psi}(h)=0$, where $h$ is a time-one map of the flow generated by some Morse function $H\colon\B \Sigma_g\to\B R$. Recall that by definition we have
$$\overline{\Psi}(h)=\int\limits_{\B \Sigma_g} \lim_{p\to\infty}\frac{\psi([h^p_x])}{p}\o\thinspace.$$
Since the set $\mathrm{Reg}_H$ is of full measure in $\B \Sigma_g$, it is enough to show that for each $x\in \mathrm{Reg}_H$ the following equality holds $\lim_{p\to\infty}\frac{|\psi([h^p_x])|}{p}=0$.

The group $\pi_1(\B\Sigma_g)$ admits the following presentation
\begin{equation}\label{eq:fund-gp-presentation}
\pi_1(\B\Sigma_g)=\langle\a_i,\b_i|\hspace{2mm} 1\leq i\leq g,\thinspace \prod_{i=1}^g [\a_i,\b_i]=1\rangle .
\end{equation}
For every $\a\in\pi_1(\B\Sigma_g)$ denote by $l(\a)$ the word length of $\a$ with respect to the set of generators given in \eqref{eq:fund-gp-presentation}. Since $\psi\in Q_{\mathcal{MCG}_g^1}(\pi_1(\B\Sigma_g), S_g)$, it is constant on conjugacy classes and thus $\psi(\a_i)=\psi(\b_i)=0$ for each  $1\leq i\leq g$. In addition, for every $\a\in\pi_1(\B\Sigma_g)$ we have $|\psi(\a)|\leq D_{\psi}l(\a)$. It follows from \eqref{eq:braid-morse-aut} that for every $p\in\B N$ and $x\in\mathrm{Reg}_H$ we have
$$
[h^p_x]=\a'_{p,x}\circ[\gamma_{x}]^{k_{h,p}}\circ\a''_{p,x}\thinspace,
$$
where $k_{h,p}$ is an integer which depends only $h$, $p$ and $x$, and $l(\a'_{p,x})\thinspace, l(\a''_{p,x})$ are bounded by some constant $K>0$  independent of $h$, $x$ and $p$. Hence for every $p\in\B N$ and $x\in\mathrm{Reg}_H$ we have
$$0\leq\frac{|\psi([h^p_x])|}{p}\leq\frac{|\psi(\a'_{p,x})|+|k_{h,p}||\psi([\gamma_{x}])|+|\psi(\a''_{p,x})|+2D_{\psi}}{p}\hspace{2mm}.$$
Since $\psi\in Q_{\mathcal{MCG}_g^1}(\pi_1(\B\Sigma_g), S_g)$, by definition it extends to a homogeneous quasi-morphism on $\mathcal{MCG}_g^1$ and vanishes on the set $S_g$. Hence by Proposition \ref{P:finite-conj-classes} we have $\psi([\gamma_{x}])=0$, and hence
$$0\leq\frac{|\psi([h^p_x])|}{p}\leq\frac{2K\cdot D_{\psi}+2D_{\psi}}{p}=\frac{2D_{\psi}(K+1)}{p}\hspace{2mm}.$$
By taking $p\to\infty$ we conclude the proof of the proposition.
\end{proof}

Since the map $\mathfrak{Polt}_g$ is injective and by Corollary \ref{cor:inf-dim-qm-restriction} the linear space $Q_{\mathcal{MCG}_g^1}(\pi_1(\B\Sigma_g), S_g)$ is infinite dimensional, an immediate consequence of Proposition \ref{P:key-proposition} is the following

\begin{cor}\label{cor:inf-dim-aut-qm}
The space $Q(\Ham(\B \Sigma_g),\Aut)$ is infinite dimensional.
\end{cor}

\begin{rem}\label{R:unbounded}
The existence of a non-trivial
homogeneous quasi-morph\-ism $\overline{\Psi}\colon \Ham(\B\Sigma_g)\to \B R$, which is trivial on the set
$\OP{Aut}\subset \Ham(\B\Sigma_g)$ of all autonomous diffeomorphisms, implies that the autonomous norm is unbounded. Indeed,
for every $f\in \Ham(\B\Sigma_g)$ we have that
$$|\overline{\Psi}(f)|=|\overline{\Psi}(h_1\circ\ldots\circ h_m)|\leq mD_{\overline{\Psi}}$$
and hence for every
natural number $n$ we get $\|f^n\|_{\OP{Aut}}\geq \frac{|\overline{\Psi}(f)|}{D_{\overline{\Psi}}}n>0$,
provided $\overline{\Psi}(f)\neq 0$.
\end{rem}
Remark \ref{R:unbounded} and Corollary \ref{cor:inf-dim-aut-qm} conclude the proof of the theorem.
\qed

\subsection*{Acknowledgments}
The author would like to thank Mladen Bestvina, Danny Calegari, Louis Funar, Sergei Lanzat, Juan Gonzalez-Meneses, Jarek Kedra, Chris Leininger and Dan Margalit for fruitful discussions. Part of this work has been done during the author's stay in Mathematisches For\-schung\-sinstitut Oberwolfach and in Max Planck Institute for Mathematics in Bonn. The author wishes to express his gratitude to both institutes. He was supported by the Oberwolfach Leibniz fellowship and Max Planck Institute research grant.

\bigskip

Department of Mathematics, Vanderbilt University, Nashville, TN\\
\emph{E-mail address:} \verb"michael.brandenbursky@vanderbilt.edu"

Max-Planck-Institut f$\ddot{\textrm{u}}$r Mathematik, 53111 Bonn, Germany\\
\emph{E-mail address:} \verb"brandem@mpim-bonn.mpg.de"

\end{document}